\documentclass[14pt]{article}
\usepackage{amsmath}
\usepackage{amssymb}
\usepackage{mathrsfs}
\usepackage{amscd}
\usepackage{amsthm,color}
\usepackage[matrix,frame,import,curve]{xypic}
\usepackage{authblk}

\newtheorem{lemma}{Lemma}[section]
\newtheorem{theorem}[lemma]{Theorem}
\newtheorem{corollary}[lemma]{Corollary}
\newtheorem{conjecture}[lemma]{Conjecture}
\newtheorem{prop}[lemma]{Proposition}
\theoremstyle{definition}
\newtheorem{definition}[lemma]{Definition}

\newtheorem{remark}[lemma]{Remark}

\theoremstyle{remark}
\newtheorem*{proof*}{Proof}
\numberwithin{equation}{section}

\def\Spec{{\bf {Spec}}}
\def\RHom{{\mathbf{R}\rm{Hom}}}

\def\D{\mathrm{D}}

\def\Ext{{\mathrm{Ext}}}

\def\Hom{{\mathrm{Hom}}}

\def\End{{\mathrm{End}}}

\def\Spec{{\mathrm{Spec\ }}}
\def\deg{{\mathrm{deg}}}
\def\perf{{\mathfrak{Perf}}}

\def\PP{{\mathbb P}}

\def\ZZ{{\mathbb Z}}
\def\CC{{\mathbb C}}

\def\BB{{\mathbb B}}

\def\bR{{\bf{R}}}

\def\cO{{\cal{O}}}

\def\cL{{\cal{L}}}

\def\cN{{\cal{N}}}

\def\cT{{\cal{T}}}
\def\cU{{\cal{U}}}

\def\fg{{\mathfrak{g}}}
\def\fh{{\mathfrak{h}}}

\def\red{{\mathrm{red}}}
\def\st{{\mathrm{st}}}

\def\deg{{\mathrm{deg}}}
\def\MF{{\mathrm{MF}}}

\def\CMR{{\text{\underline{CM}~R}}}
\def\ep{{\epsilon}}
\def\Sym{{\mathrm{Sym}}}
\def\Ac{{A_{\rm{con}}}}
\def\Bc{{B_{\rm{con}}}}

\title{Contraction algebra and singularity of three-dimensional flopping contraction}
\date{}
\author[1]{Zheng Hua\thanks{huazheng@maths.hku.hk}}
\affil[1]{Department of Mathematics, the University of Hong Kong}

\begin{document}
\maketitle
\begin{abstract}
In \cite{DW13}, Donovan and Wemyss introduced the contraction algebra of flopping curves in 3-folds. They conjectured that the contraction algebra determines the formal neighborhood of the underlying singularity of the contraction.
In this paper, we prove that the contraction algebra together with its natural $A_\infty$-structure constructed in \cite{HT16}, determine 
the formal neighborhood of the singularity. 
\end{abstract}

\section{Introduction}\label{sec:intro}
In ~\cite{DW13}, \cite{DW15}, 
Donovan and Wemyss constructed certain algebras called 
\emph{contraction algebras} associated with birational morphisms $f:X\to Y$ with at most one-dimensional fibers. The contraction algebras pro-represent the functors of non-commutative deformations of reduced exceptional fiber of $f$. It has remarkable applications to the study of autoequivalences of derived categories \cite{DW13}, birational geometry \cite{DW15} and  enumerative geometry \cite{HT16}.

When $f$ is a 3-fold flopping contraction, the contraction algebra generalizes various classical invariants of rigid rational curves, including Reid's width for 
curves with normal bundle $\cO\oplus\cO(-2)$, the length of the scheme theoretical exceptional fiber of $f$ and Katz's genus zero Gopakumar-Vafa invariants \cite{HT16}. Indeed, Donovan and Wemyss conjectured that 
the contraction algebra
recovers the formal neighborhood at the 
flopping curve, so that the contraction 
algebra contains enough information for the local 
geometry of the flopping curve.
\begin{conjecture}\emph{(\cite[Conjecture~1.4]{DW13})}\label{conjDW}
 Suppose that $X\to Y$ and $X^\prime\to Y^\prime$ are 3-dimensional flopping contractions of smooth quasi-projective 3-folds $X$ and $X^\prime$, to threefolds $Y$ and $Y^\prime$ with isolated singular points $p$ and $p^\prime$ respectively.   To these, associate the contraction algebras $\Ac$ and $\Bc$. Then the completions of stalks at $p\in Y$ and $p^\prime\in Y^\prime$ are isomorphic if and 
only if $\Ac\cong \Bc$ as algebras.
\end{conjecture}\label{conjDW}
In a joint  work with Toda \cite{HT16}, we prove that the contraction algebra is equipped with a natural $\ZZ/2$-graded $A_\infty$-structure.
In this paper, we prove a modified version of the above conjecture with the contraction algebra replaced by its $A_\infty$-enhancement. 

\begin{theorem}(Theorem \ref{mainthm})
 Suppose that $X\to Y$ and $X^\prime\to Y^\prime$ are 3-dimensional flopping contractions of smooth quasi-projective 3-folds $X$ and $X^\prime$, to threefolds $Y$ and $Y^\prime$ with isolated singular points $p$ and $p^\prime$ respectively.    Then the completions of stalks at $p\in Y$ and $p^\prime\in Y^\prime$ are isomorphic if and 
only if $\Ac$ and $\Bc$ are Morita equivalent as $\ZZ/2$-graded $A_\infty$-algebras.
\end{theorem}

A famous theorem of Mather and Yau (Theorem \ref{MY}) claims that (the germ of) an isolated hypersurface singularity is determined by its Tjurina algebra (see Section 4 for the definition). For a 3-dimensional flopping contraction $f:X\to Y$, it is well known that $Y$ has isolated hypersurface singularities (see \cite{Re83}). Therefore, one possible way to prove Conjecture \ref{conjDW} is to show that the contraction algebra recovers the Tjurina algebra. 

Our approach can be summarized as follows. First, it is proved by Dyckerhoff that the Hochschild cohomology of the derived category of singularities $\D_{sg}(Y)$ is isomorphic with the Milnor algebra of $Y$ (\cite{Dyc09}). It follows from Proposition \ref{generator} and Theorem \ref{inftycontr} that the contraction algebra together with its canonical $A_\infty$-structure is Morita equivalent with $\D_{sg}(Y)$. So the Hochschild cohomology of the $A_\infty$ contraction algebra is isomorphic with the Milnor algebra of $Y$. Finally, we prove that the class in the Milnor algebra, represented by the defining equation of $Y$, is precisely the class in the Hochschild cohomology represented by the $A_\infty$-products on the contraction algebra. To prove this, we need to use a result of Efimov \cite{Efi09} that computes the $A_\infty$-structure on the minimal model of the endomorphism dg-algebra for the structure sheaf of the singular point. Meanwhile, we use the dg-algebras of upper-triangular matrices (see definition in Section \ref{subsec:HHAinfty}), due to Happel \cite{Ha89}, Buchwicz \cite{Bu03} and Keller \cite{Keller03} to show that the class in the Hochschild cohomology represented by the $A_\infty$-products is invariant under Morita equivalences. This leads to a proof of Theorem \ref{mainthm}.

Theorem \ref{mainthm} is weaker than Conjecture \ref{conjDW}. It is not clear whether the algebra structure alone is sufficient to determine the formal neighborhood of the singular point. We observe that if one considers the Hochschild cohomology of the contraction algebra (without the $A_\infty$-structure) then the answer is an infinite dimensional space except for the Atiyah flop case. This seems to suggest that if one wants to recover the singularity via its Milnor algebra (and Tjurina algebra). Then the $A_\infty$-structure is crucial. However, we do not have a counter-example to Conjecture \ref{conjDW} so far.

The paper is organized as follows. In Section \ref{sec:pre}, we recall the construction of contraction algebra  and its $A_\infty$-structure. The situation here is slightly more general than that in \cite{HT16}, where the reduced fiber is assumed to be smooth and irreducible.
In Section \ref{sec:HH}, we survey several results on Hochschild cohomology and Morita theory of $A_\infty$-algebras that are needed for the proof of the main theorem. Several results on Hochschild cohomology of category of matrix factorizations are summarized in Section \ref{subsec:HHmf}, including the lemma of Efimov (Lemma \ref{lemma-Efi}). In Section \ref{sec:proof}, we give the proof of the main result. The readers who are familiar with the theory of Hochschild cohomology of $A_\infty$-algebras can skip Section 3.1 and 3.2.

\paragraph{Acknowledgments.} We are grateful to Will Donovan and Michael Wemyss for many valuable discussions. The research was supported by RGC Early Career grant no. 27300214 and NSFC Science Fund for Young Scholars no. 11401501.

\section{Contraction algebra and its $A_\infty$-structure}\label{sec:pre}
\subsection{Contraction algebra}
Let $X$ be a smooth quasi-projective complex 3-fold.  
A flopping contraction is a birational morphism 
\begin{equation}\label{flop}
f: X\to Y
\end{equation}
which is an isomorphism in codimension one. We have ${\bf{R}}f_*\cO_X=\cO_Y$ and $Y$ has Gorenstein terminal singularities. Three-dimensional Gorenstein singularities that admit small resolutions have been classified by Reid \cite{Re83}. By Corollary 1.12 \cite{Re83}, they are compound Du Val singularities, which means a generic hyperplane section of $Y$ through the singular point has Du Val singularity. In particular, compound Du Val singularities are hypersurface singularities. Therefore, when a 3-dimensional flopping contraction $f:X\to Y$ is discussed we may always assume that $Y$ is a hypersurface in $\CC^4$. 

We denote the exceptional locus of $f$ by $C$ and its reduced scheme by $C^{\red}$. 
It is well known that $C^{\red}$ has a decomposition
 \[
 C^{\red}=\bigcup_{i=1}^n C_i^{\red}.
 \]
where $C^{\red}_i\cong \PP^1$.

Let $p \in Y$ be the image of $C$ under $f$, 
and we set $R=\hat{\cO}_{Y,p}$ the formal completion of $\cO_Y$ at the singular point $p$. 
We take the completion of \ref{flop}
\begin{equation}\label{compflop}
\hat{f}:\hat{X}:=X\times_Y \Spec R\to \hat{Y} :=\Spec R.
\end{equation}
Let $\cL_i$ be a line bundle on $\hat{X}$ such that $\deg_{C_j}\cL_i=\delta_{ij}$. Define $\cN_i$ to be given by the maximal extension 
\begin{equation}\label{extN}
\xymatrix{
0\ar[r] & \cL_i^{-1}\ar[r] & \cN_i\ar[r] &\cO_{\hat{X}}^{\oplus r_i}\ar[r] &0
}
\end{equation} associated to a minimal set of $r_i$ generators of $H^1(\hat{X},\cL_i^{-1})$. 
Denote $\bigoplus_{i=1}^n \cN_i$ by $\cN$.

We set $\cU:= \cO_{\hat{X}}\oplus \cN$, $N:=\bR f_*\cN=f_*\cN$ and define
\[
A:= \End_{\hat{X}}(\cU).
\]
By Lemma 4.2.1 \cite{VdB04}, $A\cong \End_{R}(R\oplus N)$. 
Van den Bergh (\cite[Section~3.2.8]{VdB04}) showed that $\cU$ is a tilting object, i.e. $F:=\RHom_{\hat{X}}(\cU,-)$ defines an equivalence of triangulated categories $\D^b(\rm{coh}(\hat{X}))\cong \D^b(\text{mod-}A)$.

\begin{definition}(\cite[Definition~2.8]{DW13})  \label{A_con}
Let $R$, $N$ and $A$  be those defined in \ref{extN}.
The contraction algebra $\Ac$ is defined to be $A/I_{\rm{con}}$, where $I_{\rm{con}}$ is the two sided ideal of $A$ consisting of morphisms $R\oplus N\to R\oplus N$ factoring through a summand of finite sums of $R$.
\end{definition}

\begin{remark}
In \cite{DW13} and \cite{DW15}, Donovan and Wemyss gave an alternative definition of the contraction algebra as the algebra pro-representing the non-commutative deformations of $\bigoplus_{i=1}^n\cO_{C^{\red}_i}$. It is equivalent with Definition \ref{A_con} under the setting of this paper. We refer to Theorem 3.9 of \cite{DW15} for the proof of this statement. 
\end{remark}

\subsection{$A_\infty$-structure on $\Ac$}

Let $Y$ be a quasi-projective scheme. Denote $\D^b({\rm{coh}}(Y))$ for the bounded derived category of coherent sheaves on $Y$ and $\perf(Y)$ for the full subcategory consisting of perfect complexes on $Y$. We define a triangulated category $\D_{sg}(Y)$ as the quotient of $\D^b({\rm{coh}}(Y))$ by $\perf(Y)$ 
(cf.~\cite[Definition~1.8]{Orlov02}). Consider the case when $Y$ is a hypersurface in a smooth affine variety $\Spec B$ defined by $W=0$ for a function $W\in B$. Denote $B/(W)$ by $R$.
Buchweitz, and independently Orlov,  proved that the derived category of singularities $\D_{sg}(Y)$ is equivalent, as triangulated categories, with the stable category of maximal Cohen-Macaulay $R$-modules $\CMR$, which is equivalent with the homotopy category of the dg-category of matrix factorizations $\text{MF}(W)$ (cf.~\cite[Theorem~3.9]{Orlov02}). The objects of $\text{MF}(W)$ are the ordered pairs:
\[\bar{E}=(E,\delta_E):=\xymatrix{
E_1\ar@/^1mm/[r]^{\delta_1} &E_0\ar@/^1mm/[l]^{\delta_0} 
}
\] where $E_0$ and $E_1$ are finitely generated projective $B$-modules  and the compositions $\delta_0\delta_1$ and $\delta_1\delta_0$ are the multiplications by the element $W\in B$. 
The precise definition of the category of matrix factorizations $\text{MF}(W)$ can be found in~\cite[Section~3]{Orlov02}.

Below, we assume that $Y$ is an affine hypersurface with one isolated singular point. 
By taking the completion at the singular point, we may further assume that $B=\CC[[x_1,\ldots,x_n]]$ and $W$ has an isolated critical point at the origin. We define the \emph{Milnor algebra} of $Y$ to be $M_W:=\CC[[x_1,\ldots,x_n]]/J_W$ with $J_W$ being the ideal $(\partial_1 W,\ldots, \partial_n W)$.

Let $\cT$ be a triangulated category admitting infinite coproducts. An object $X$ in $\cT$ is called \emph{compact} if the functor $\Hom_\cT(X,-)$ commutes with infinite coproducts. We call $X$ a generator of $\cT$ if the right orthogonal complement of $X$ is equivalent to 0. The category $\MF(W)$ embeds in the category $\MF^\infty(W)$, consisting of matrix factorizations of possibly infinite rank.
In \cite{Dyc09}, Dyckerhoff proved that the structure sheaf of the singular point is a compact generator of  $\MF^\infty(W)$ (Theorem 4.1 \cite{Dyc09}). And $\MF(W)$ consists of compact objects in $\MF^\infty(W)$. 
As a consequence, the Hochschild cohomology of $\text{MF}(W)$ is isomorphic, as $\CC$-algebras, with $M_W$ (Corollary 6.4 of \cite{Dyc09}).

Let $\hat{f}:\hat{X}\to \hat{Y}=\Spec R$ be a 3-dimensional flopping contraction considered in the previous section, where $R= \CC[[x_1,\ldots,x_4]]/(W)$. A coherent sheaf $E$ on $\hat{Y}$ can be identified with a finitely generated $R$-module $M$. It represents an object in the category $\D_{sg}(\hat{Y})$. We denote the corresponding matrix factorization of $M$ by $M^{\st}$, called the \emph{stabilization} of $M$.

From the theorem of Dyckerhoff (Theorem 4.1 \cite{Dyc09}), we know that the stabilization of the structure sheaf of the origin, denoted by $k^{\st}$, is a compact generator of $\MF^\infty(W)$. For the singularity underlying a 3-dimensional flopping contraction, a different generator exists by  a theorem of Van den Bergh  (\cite[Section~3.2.8]{VdB04}) and the work of Iyama and Wemyss \cite{IW10}.
\begin{prop}\label{generator}
Let $\cU=\cN\oplus\cO_{\hat{X}}$ be the tilting bundle on $\hat{X}$ and $N$ be the $R$-module defined as $\hat{f}_* \cN$.
Its stabilization $N^{\st}$ is a compact generator of ${\MF^\infty}(W)$.
\end{prop}
\begin{proof}
Because $R\oplus N$ is a tilting object, proposition 5.10 of \cite{IW10} implies that $N$ is a generator. By~\cite[Section~4]{Dyc09}, $N^{\st}$ is a compact object.
\end{proof}
In \cite{HT16}, we prove that the contraction algebra is isomorphic with the endomorphism algebra of $N^{\st}$ in $\MF(W)$. Therefore it carries a natural $\ZZ/2$-graded $A_\infty$-structure.
\begin{prop}(\cite[ Proposition 3.2]{HT16})\label{inftycontr}
Let $W,R,N$ be defined as above. We have 
\begin{align}\label{Acon:id}
\Ac\cong \RHom_{\MF(W)}(N^{\st},N^{\st}). 
\end{align}
\end{prop}
\begin{proof}
From the definition of the contraction algebra, $\Ac$ is the ring of endomorphism of $N$ modulo the ideal of elements that factors through a projective $R$-module. This is precisely the definition of endomorphism in the stable category of maximal Cohen-Macaulay modules. By the equivalence $\CMR\cong \MF(W)$
\[
\Ac\cong \Hom^0_{\MF(W)}(N^{\st},N^{\st}).
\]
Because $\MF(W)$ is $\ZZ/2$-graded, it suffices to show that $\Hom^1_{\MF(W)}(N^{\st},N^{\st})$ vanishes. Since $N$ is a maximal Cohen-Macaulay $R$-module, it suffices to prove that $\Ext^1_R(N,N)$ vanishes. Because $N=f_*\cN$, $\Hom^0_R(N,N)$ is a maximal Cohen Macaulay module and $\Ext^1_R(N,N)$ is supported at the singular point. By Lemma 2.7 of \cite{IW10}, $\Ext^1_R(N,N)$ vanishes.
\end{proof}
Note that the $A_\infty$-products $m_k$ on $\Ac$ vanishes for odd $k$.
We may consider the $A_\infty$-enhanced version of the conjecture of Donovan and Wemyss. The proof will be presented in Section \ref{sec:proof}.
\begin{theorem}(\cite[Conjecture 5.3]{HT16})\label{mainthm}
Suppose that $X\to Y$ and $X^\prime\to Y^\prime$ are 3-dimensional flopping contractions of smooth quasi-projective 3-folds $X$ and $X^\prime$, to threefolds $Y$ and $Y^\prime$ with isolated singular points $p$ and $p^\prime$ respectively.    Then the completions of stalks at $p\in Y$ and $p^\prime\in Y^\prime$ are isomorphic if and 
only if $ A_{con}$ and $ B_{con}$ are Morita equivalent as $\ZZ/2$-graded $A_\infty$ algebras.
\end{theorem}

\section{Hochschild cohomology}\label{sec:HH}
In this section, we collect several definitions and results about Hochschild cohomology of graded algebras, and more generally $A_\infty$-algebras. The content of Section 3.1 and 3.2 are well known to experts. We present them here just for our conveniences. Our main references are \cite{Ger62} for Section 3.1, \cite{Keller01}, \cite{Keller03} and \cite{Mer16} for Section 3.2. In Section 3.3, we recall the definitions of polyvector fields and Schouten bracket following \cite{Efi09}. Lemma \ref{lemma-Efi} is due to Efimov (Lemma 8.2 \cite{Efi09}). This is the key result  that allows us to relate $W$ to the $A_\infty$-structure of $\Ac$.
\subsection{Hochschild cohomology of graded algebras}\label{subsec:HHgraded}
Let $A$ be a $\ZZ$-graded or $\ZZ/2$-graded algebra over a field $k$. 
As a convention, we denote an element of $A$ by $a_i$ and its degree by $|a_i|$. Denote a map from $A\otimes\ldots\otimes A\to A$ by $f$ and its degree by $|f|$.

The \emph{Hochschild cochain complex} $CC^\bullet(A,A)$ is defined as
\begin{equation}\label{HHcochain}
CC^d(A,A)=\prod_{i+j=d} \Hom_k(A^{\otimes i},A[j]).
\end{equation}
Here $\Hom_k$ stands for the space of $k$-linear maps of degree zero. And $A[j]$ is the free $A$-module defined by $A[j]_l:=A_{l+j}$. 
The \emph{Hochschild differential} $\partial$ is defined by
\begin{align}\label{HHd_alg}
\partial(f) (a_1,\ldots,a_n)&=-(-1)^{(|a_1|+1)|f|}a_1f(a_2,\ldots,a_n)\notag\\
&-\sum_{i=2}^n(-1)^{\ep_i} f(a_1,\ldots,a_{i-1}a_i,\ldots,a_n)\\
&+(-1)^{\ep_n}f(a_1,\ldots,a_{n-1})a_n,\notag
\end{align}
where $\ep_i=|f|+|a_1|+\ldots+|a_{i-1}|-i+1$. The \emph{Hochschild cohomology} of the graded algebra $A$, denoted by $HH^\bullet(A,A)$ is defined to be 
$H(CC^\bullet(A,A),\partial)$.

 For $f\in CC^n(A,A)$ and $g\in CC^m(A,A)$, the cup product on $CC^\bullet(A,A)$ is defined by
\[
(f\cup g)(a_1,\ldots,a_{m+n})=(-1)^{|g|(\sum_{i\leq m}|a_i|+1)}f(a_1,\ldots,a_n)g(a_{n+1},\ldots,a_{n+m}).
\]
 For $f\in CC^n(A,A)$ and $g\in CC^m(A,A)$ one defines the composition at the i-th place (or the \emph{braces} structure) $f\circ_i g\in CC^{m+n-1}(A,A)$
\begin{align}\label{braces}
f\circ_i g(a_1,\ldots,a_{m+n-1})=(-1)^{(|g|+1)\sum_{j\leq i}(|a_j|+1)}f(a_1,\ldots,a_{i-1},g(a_i,\ldots,a_{m+i-1}),\ldots,a_{m+n-1}).
\end{align}
Define
\begin{align*}
f\circ g :=\sum_{i=1}^n f\circ_i g; & &  \{f,g\}:=f\circ g-(-1)^{(n-1)(m-1)}g\circ f.
\end{align*}

The bracket $\{~,~\}$ induces the \emph{Gerstenhaber bracket} on $HH^\bullet(A,A)$. It was proved in \cite{Ger62} that $\{,\}$ and $\cup$ satisfy the properties:
\begin{enumerate}\label{HHoddPoisson}
\item[(1)] $\{f,g\}=-(-1)^{(|f|-1)(|g|-1)}\{g,f\}$;
\item[(2)] $\{f,g\cup h\}=\{f,g\}\cup h+(-1)^{(|f|-1)|g|}g\cup\{f,h\}$;
\item[(3)] $(-1)^{(|f|-1)(|h|-1)}\{f,\{g,h\}\}+cp(f,g,h)=0$;
\item[(4)] $\cup$ induces a (graded) commutative product on $HH^\bullet(A,A)$.
\end{enumerate}
Here $cp(f,g,h)$ means cyclic permutation of the previous term. The above properties says that the cup product and the Gerstenhaber bracket make $HH^\bullet(A,A)$ a Poisson algebra of degree $-1$. This can be viewed as an analogue of the odd Poisson structure on the space of polyvector fields on a smooth manifold.

Observe that the definition of the braces structure does not involve the algebra structure on $A$ while the differential $\partial$ and the cup product do.
The product $\mu:A\otimes A\to A$ can be viewed as an element in $CC^2(A,A)$. By formulae \ref{HHd_alg} and \ref{braces}, it is easy to check that 
\[
\partial(f)=(-1)^{|f|-1}\{\mu,f\}.
\]
The associativity of $\mu$ is equivalent with the equation $\{\mu,\mu\}=0$.

Denote $A^e$ for $A\otimes A^{\rm{op}}$, where $A^{\rm{op}}$ is the opposite algebra of $A$ defined by
\[
\mu^{\rm{op}}(a,b)=(-1)^{|a||b|}\mu(b,a).
\]
Given an $A^e$-module $M$, we may define the Hochschild cochain complex with value in $M$
\[
CC^d(A,M)=\prod_{i+j=d} \Hom_k(A^{\otimes i},M[j]).
\]
The formula of the differential is  the same with \ref{HHd_alg}, except in the first and last term the multiplications are replaced by the actions. Denote the bar resolution of $A$ as $A^e$-modules
by $\BB(A)$.
Because $\Hom_{A^e}(A^{\otimes n+2},M)\cong \Hom_k(A^{\otimes n},M)$, $CC^\bullet(A,M)$ is isomorphic to $\Hom_{A^e}(\BB(A),M)$ as complexes of $k$-modules.

Let $A$ and $B$ be two graded $k$-algebras and $X$ be a $A\otimes B^{\rm{op}}$-module. Let $C^\bullet(A,X,B)$ be a graded vector space with
\[
C^d(A,X,B):=\prod_{j+l+m=d-1} \Hom_k(A^{\otimes l}\otimes X\otimes B^{\otimes m},X[j]).
\]
The differential on $C^\bullet(A,X,B)$ is defined by a similar formula like \ref{HHd_alg} with the product replaced by action when an element of $X$ is involved.

\subsection{Hochschild cohomology of $A_\infty$-algebra}\label{subsec:HHAinfty}
The content of this subsection is essentially due to Keller.
In \cite{Keller03}, Keller proved that the Hochschild cohomology of dg-category is invariant under Morita equivalence. In this subsection, we apply Keller's construction to the case when the Morita equivalence is given by choosing two different generators of the dg-category. Because we need to keep track of certain class in the Hochschild cohomology under this equivalence in our application, it is more convenient to restate Keller's construction in the language of $A_\infty$-algebras.  

Let $A$ be a $\ZZ$-graded or $\ZZ/2$-graded $A_\infty$-algebra over $k$, with $A_\infty$-structures
\[
m_k: A^{\otimes k}\to A[2-k]~~\text{for}~~k\geq 1.
\]
For $n\geq 1$, we have
\begin{align}\label{Ainftyrel}
\sum_{r+s+t=n}(-1)^{r+st}m_u(a_1,\ldots,a_r,m_s(a_{r+1},\ldots,a_{r+s}),a_{r+s+1},\ldots,a_n)=0
\end{align}
where we put $u=r+1+t$.

The element $m:=\sum_{k\geq 1}m_k$ defines an element in $CC^2(A,A)$. The following result follows from a straightforward calculation.
\begin{prop}
The equation $\{m,m\}=0$ is equivalent with the relations \ref{Ainftyrel}.
\end{prop}
We define the differential on $CC^\bullet(A,A)$ by 
\begin{align}\label{d-Ainfty}
\partial(f)=(-1)^{|f|-1}\{m,f\}.
\end{align}
Similar to the graded algebra case, we define $HH^\bullet(A,A)$ to be $H(CC^\bullet(A,A),\partial)$.
Clearly, $m$ represents a canonical cohomology class in $HH^2(A,A)$. Given an $A_\infty$-algebra $A$ with $m\in CC^2(A,A)$, an element $m^\prime\in CC^2(A,A)$ is said to satisfy the \emph{Maurer-Cartan} equation if
\[
\{m+m^\prime,m+m^\prime\}=0.
\]
It is easy to check that $m+m^\prime$ defines an $A_\infty$-structure if and only if the above equation holds. We call the new differential $\partial_{m^\prime}$ defined by
\[
\partial_{m^\prime}(f):=\partial f+\{m^\prime,f\}
\] \emph{the twisted differential} defined by $m^\prime$.

An $A_\infty$-module over $A$ is a graded space $M$ endowed with maps:
\[
m_k^M:A^{\otimes k-1}\otimes M\to M[2-k], ~~k\geq 1
\]
such that an identity of the form \ref{Ainftyrel} holds, but with $m_u(\ldots,a_r,m_s,a_{r+s+1}\ldots)$ replaced by $m^M_u(\ldots,a_r,m_s,a_{r+s+1}\ldots,a_{r+s+t})$ for $t>0$ and 
$m^M_u(\ldots,m^M_s)$ for $t=0$.

Let $A$ and $B$ be $A_\infty$-algebras. An $A_\infty$-module over $A\otimes B^{\rm{op}}$ (or an $A$-$B$ $A_\infty$-bimodule) $X$ is a graded space endowed with maps
\[
m_{l,m}^X: A^{\otimes l}\otimes X\otimes B^{\otimes m}\to X,
\]
such that for $n\geq 1$, $m^X_{n,0}$ and $m^X_{0,n}$ satisfies the relations of $A_\infty$-modules over $A$ and $B^{\rm{op}}$ respectively, and
\begin{align*}
&\sum_{\substack{l+m+s=n-1\\ r+t=s}}\pm m^X_{l,m}(a_1,\ldots,a_l,m^X_{r,t}(a_{l+1},\ldots,a_{l+r},x,a_{l+r+1}\ldots,a_{l+r+t}),a_{l+s+1},\ldots,a_{n-1})\\
&+\sum_{\substack{r+s+t+m=n-1\\ r+t+1=l}}\pm m^X_{l,m}(a_1,\ldots,a_r,m^A_s(a_{r+1},\ldots,a_{r+s}),a_{r+s+1},\ldots,a_{r+s+t},x,\ldots,a_{n-1})\\
&+\sum_{\substack{r+s+t+l=n-1\\ r+t+1=m}}\pm m^X_{l,m}(a_1,\ldots,a_l,x,a_{l+1},\ldots,a_{l+r},m^B_s(a_{l+r+1},\ldots,a_{l+r+s}),a_{l+r+s+1},\ldots,a_{n-1})\\ &=0
\end{align*}
for $l,m\neq 0$. The choice of the sign is rather subtle. Since we will not use it in this paper, we refer to Section 1 of \cite{Mer16} for the details about the sign.

Given $A_\infty$-algebras $A$, $B$ and a $A\otimes B^{\rm{op}}$-module $X$, we can construct a new $A_\infty$-algebra $G$ consisting of upper-triangular matrices
\[g:=\left(\begin{array}{cc} a&x\\0&b\end{array}\right)\]
for $a\in A$, $b\in B$ and $x\in X$. The $A_\infty$-structure on $G$ is defined by
\[
m_k^G=\left(\begin{array}{cc} m_k^A&\sum_{l+m=k-1} m^X_{l,m}\\0&m_k^B\end{array}\right).
\]
Again, we denote $m^G$ for $\sum_k m^G_k$, $m^A$ for $\sum_k m^A_k$ and $m^B$ for $\sum_k m^B_k$.

There are canonical inclusions of complexes of $G^e$-modules 
\[
\mu_A: A\otimes_A \BB(A)\otimes_A (A\oplus X)\to \BB(G)
\] and 
\[
\mu_B: (B\oplus X)\otimes_B \BB(B)\otimes_B B\to \BB(G) 
\] defined as follows.
For $n\geq 0$, we have
\[
A\otimes_A A^{\otimes (n+2)}\otimes_A(A\oplus X)\cong A\otimes_k A^{\otimes n}\otimes_k (A\oplus X).
\]
The right hand side embeds into $G^{\otimes (n+2)}$ as a summand. This defines the inclusion $\mu_A$, and similarly $\mu_B$ when $A$ is replaced by $B$. 

Observe that 
\[\Hom_{G^e}(A\otimes_A \BB(A)\otimes_A (A\oplus X), G)\cong \Hom_{A^e}(\BB(A),A)\]
by adjunction. As a consequence, $\mu_A$ induces a chain map
\[
\mu_A^*:CC^\bullet(G,G)=\Hom_{G^e}(\BB(G),G)\to CC^\bullet(A,A)=\Hom_{A^e}(\BB(A),A).
\]
Similarly, $\mu_B$ induces a map $\mu^*_B: CC^\bullet(G,G)\to CC^\bullet(B,B)$.

The following proposition follows immediately from the definition of $\mu_A$ and $\mu_B$.
\begin{prop}\label{mumap}
\[
\mu^*_A(m^G)=m^A,~~\mu^*_B(m^G)=m^B.
\]
\end{prop}

Let $A$ and $B$ be two $A_\infty$-algebras and $X$ be a $A\otimes B^{\rm{op}}$-module. Similar to Section \ref{subsec:HHgraded}, we define the complex $C^\bullet(A,X,B)$ with
\[
C^d(A,X,B):=\prod_{j+l+m=d-1} \Hom_k(A^{\otimes l}\otimes X\otimes B^{\otimes m},X[j]).
\]
and the Hochschild differential \ref{d-Ainfty} but with $m^A$ and $m^B$ replaced by $m^X$ when the inputs contains elements of $X$. 

We define a morphism of complexes
\[
\alpha: CC^\bullet(A,A)\to C^\bullet(A,X,B)
\] by sending 
\[
f\in \Hom_k(A^{\otimes i},A[d-i])
\] to a map  $\alpha(f)\in\prod_{l}\Hom_k(A^{\otimes l}\otimes X,X[d-l])$ with
\[
\alpha(f)=\sum_{r+i+t=l} m^X_{l-i+2}(a_1,\ldots,a_r,f(a_{r+1},\ldots,a_{r+i}),a_{r+i+1},\ldots,a_{r+i+t},x).
\]
Similarly, we define a morphism
\[
\beta: CC^\bullet(B,B)\to C^\bullet(A,X,B).
\]
The above maps $\alpha,\beta$ are special cases of a more general construction of Keller (see Section 4.4 \cite{Keller03}).
\subsection{Hochschild cohomology of category of matrix factorizations}\label{subsec:HHmf}
In this subsection, we recall an important lemma of Efimov (Lemma 8.2 \cite{Efi09}), which compares the $L_\infty$-structures on polyvector fields and the $L_\infty$-structures on Hochschild cochain complex of exterior algebras. We will follow Section 3 and 8 of \cite{Efi09}.

Let $V$ be a finite dimensional $k$-vector space with $k=\CC$ and $W$ be an element in $\Sym^{\geq 2} (V^\vee)$. We will further assume that $0$ is an isolated critical point of $W$. We now construct the matrix factorization of the structure sheaf of origin. 

Decompose $W$ into graded components 
\[
W=\sum_{i\geq 2} W_i,~~W_i\in\Sym^i (V^\vee).
\]
Define an one-form 
\[
\omega=\sum_{i\geq 2} \frac{dW_i}{i}.
\]
Denote the basis of $V$ by $\{\xi_k\}$ and the dual basis on $V^\vee$ by $\{z_k\}$. We may identify $\xi_k$ with the constant vector field on $V$. Denote the Euler vector field $\sum_k z_k\xi_k$ on $V$ by $\eta$. It is easy to check that $\iota_\eta \omega=W$, where $\iota$ is the contraction by vector field. Now the $\ZZ/2$-graded vector space $\Sym (V^\vee)\otimes \Lambda (V^\vee)$ with the odd operator $\delta=\iota_\eta+\omega\wedge~$ defines a matrix factorization for the structure sheaf of the origin. We denote its corresponding object in $\MF(W)$ by $k^{\st}$.

The endomorphism space $B_W:=\Hom_{\MF(W)}(k^{\st},k^{\st})$ is endowed with a structure of $\ZZ/2$-graded dg-algebra. If we take $W=0$, then $\delta$ defines the Koszul resolution of the structure sheaf of origin in $V$, and $B:=B_0$ is quasi-isomorphic to the exterior algebra $\Lambda(V)$. The dg-algebra $B_W$ is a deformation of the dg-algebra $B$. Therefore, the minimal model of $B_W$ is an $A_\infty$-deformation of the exterior algebra $\Lambda(V)$.

Define $\ZZ/2$-graded dg-Lie algebras $\fg$ and $\fh$ by the formulae: 
\[
\fg^d=\Sym(V^\vee)\otimes\Lambda^d(V),
\] and
\[
\fh^d=\prod_{i+j=d}\Hom(\Lambda(V)^{\otimes i},\Lambda(V)[j]). 
\] 
Here $\fg$ is equipped with the zero differential and the Schouten bracket (defined below). And $\fh$ is taken to be the Hochschild cochain complex of the $\ZZ/2$-graded algebra $\Lambda(V)$.
\begin{definition}(c.f \cite[3.2]{Efi09})
Given $f\xi_{i_1}\wedge\ldots,\wedge\xi_{i_k}\in \fg^{k}$, $g\xi_{j_1}\wedge\ldots\wedge\xi_{j_l}\in\fg^l$, 
the Schouten bracket, denoted by $\{~,~\}_{sc}$, on $\fg$ is defined to be
\begin{align*}
&\{f\xi_{i_1}\wedge\ldots,\wedge\xi_{i_k},g\xi_{j_1}\wedge\ldots\wedge\xi_{j_l}\}_{sc}=\\
&\sum_{q=1}^k(-1)^{k-q}(f\partial_{i_q}g)\xi_{i_1}\wedge\ldots\wedge
\widehat{\xi_{i_q}}\wedge\ldots\wedge\xi_{i_k}\wedge\xi_{j_1}\wedge\\
&\ldots\wedge\xi_{j_l}+\sum_{p=1}^l(-1)^{l-p-1+(k-1)(l-1)}(g\partial_{j_p}f)\xi_{j_1}\wedge\ldots\wedge
\widehat{\xi_{j_p}}\wedge
\ldots\wedge\xi_{j_l}\wedge\xi_{i_1}\wedge\ldots\wedge\xi_{i_k}.
\end{align*}
\end{definition}

We recall an important lemma of Efimov.
\begin{lemma}(\cite{Efi09}{ Lemma 8.2})\label{lemma-Efi}
\footnote{Here we use the standard grading on Hochschild cochain which differs with the grading of Efimov by 1. Efimov used the pro-nilpotent version of $\fg$ and $\fh$ because he needed to define a gauge group action on the MC locus.}
There exists a $L_\infty$-quasi-isomorphism between $\fg$ and $\fh$. Under such a quasi-isomorphism, the class $W\in \fg$ corresponds to the  class $m\in \fh$, which is the $A_\infty$-structure on the minimal model of $B_W$.
\end{lemma}
\begin{remark}
The above lemma is essentially a consequence of the Koszul duality for curved $A_\infty$-algebras. The classical Koszul duality is a Morita equivalence between $\Sym(V^\vee)$ and $\Lambda(V)$ induced by the bimodule $\Lambda(V^\vee)\otimes\Sym(V^\vee)$. Consider a (curved-)$A_\infty$ deformation of $\Sym(V^\vee)$ (strictly speaking its completion $k[[V^\vee]]$) defined by $W$. The dg-category of matrix factorizations $\MF(W)$ can be defined alternatively as the dg-category of the curved $A_\infty$-algebra $(k[[V^\vee]],W)$ (see \cite{PP12}). And $B_W$ is its Koszul dual. Then the above lemma follows from Kontsevich formality and Proposition \ref{mumap} with $X$ being the Koszul $A_\infty$-bimodule $\Lambda(V^\vee)\otimes k[[V^\vee]]$. We refer to Section 8 of \cite{Efi09} for the details of the proof.
\end{remark}

For degree reasons, $\{W,W\}_{sc}=0$ for any $W\in \Sym(V^\vee)$. Given $\theta=g\xi_{j_1}\wedge\ldots\wedge\xi_{j_l}$,
\[
\{W,\theta\}_{sc}=\sum_{p=1}^l(-1)^p(g\partial_{j_p}W)\xi_{j_1}\wedge\ldots\wedge
\widehat{\xi_{j_p}}\wedge
\ldots\wedge\xi_{j_l}.
\]
\begin{remark}
Define the differential on $\fg$ by
\[
\partial:=\{W,~\}_{sc}.
\]
Then $H(\fg,\partial)$ is isomorphic to the Milnor algebra $\Sym(V^\vee)/(\partial_i W)$. Moreover, $W$ represents a zero cohomology class in $H(\fg,\partial)$ if and only if $W$ is quasi-homogeneous.

The above statement can be proved as follows. 
Suppose the dimension of $V$ is equal to $n$.
By choosing a nowhere vanishing holomorphic $n$-form on $V$, we may identify $\Lambda^i(V)$ with $\Lambda^{n-i}(V^\vee)$. Then $\fg$ can be identified, using the volume form, with the algebraic de Rham complex $\Omega_V^\bullet$ with differential $dW\wedge~$. The first part of the claim then follows from the fact that $W$ has only isolated critical points.
The class of $W$ is zero if and only if there exists a vector field $\gamma$ on $V$ such that $\{W,\gamma\}_{sc}=W$. The one parameter subgroup corresponds to $\gamma$ defines the desired $\CC^*$-action that makes $W$ quasi-homogeneous.
\end{remark}

The following corollary follows from Lemma \ref{lemma-Efi} after passing to cohomology groups of $\fg$ and $\fh$.
\begin{corollary}\label{m=W}
Let $m\in \fh$ be the $A_\infty$-structure on the minimal model of $B_W$. Denote by $d:=\{m,~\}$  the twisted differential defined by the Maurer-Cartan element $m\in\fh$. Here $\{~,~\}$ is the Gerstenhaber bracket. Then under the isomorphism $H(\fg,\partial)\cong H(\fh,d)$ constructed in Lemma \ref{lemma-Efi}, the class of $W$ corresponds to the class of $m$.
\end{corollary}

\section{Proof of the main theorem}\label{sec:proof}

Let $x_1,\ldots,x_n$ be the coordinate of $\CC^n$. Denote the stalk of the sheaf of holomorphic functions $\cO_{\CC^n,0}$ by $\CC\{x_1,\ldots,x_n\}$.
\begin{definition}(\cite[ Definition 2.9]{GLS06})
Let $f,g\in \CC\{x_1,\ldots,x_n\}$. $f$ is called \emph{contact equivalent} to $g$, $f\sim^c g$, if there exists an automorphism $\phi$ of $\CC\{x_1,\ldots,x_n\}$ and a unit $u\in \CC\{x_1,\ldots,x_n\}^*$ such that $f=u\cdot \phi(g)$.
\end{definition}
Denote the hypersurfaces defined by $f$ and $g$ by $Z_f$ and $Z_g$. Mather and Yau proved that the contact equivalence is the same as the biholomorphic equivalence for the germs of hypersurfaces.
\begin{prop}(\cite[Proposition 2.5]{BY90})\label{cont=>bihol}
Let $f,g\in \CC\{x_1,\ldots,x_n\}$. Then $f\sim^c g$ if and only if $(Z_f,0)$ and $(Z_g,0)$ are biholomorphically equivalent.
\end{prop}
For $f,g\in(x_1,\ldots,x_n)\subset\CC\{x_1,\ldots,x_n\}$ with isolated critical points, we denote the Milnor algebras by $M_f$ and $M_g$. The \emph{Tjurina algebra} $T_f$ of $f$  is defined to be the quotient algebra of $M_f$ by the ideal generated by $f$. 
The following famous result is due to Mather and Yau.
\begin{theorem}($\mathbf{Mather-Yau}$, c.f.  Theorem 2.26 \cite{GLS06})\label{MY}
Let $f,g$ be defined as above. Then $f\sim^c g$ if and only if $T_f\cong T_g$.
\end{theorem}

Combing Proposition \ref{cont=>bihol} and Theorem \ref{MY}, we get that the Tjurina algebra determines the germ of the hypersurface singularity. If we assume that $f$ is quasi-homogeneous, then $f$ represents a zero class in the Milnor algebra $M_f$. In this case, the Milnor algebra $M_f$ and the Tjurina algebra $T_f$ are isomorphic. Therefore, the germ of the quasi-homogeneous hypersurface singularity is determined by its Milnor algebra.

We recall an important lemma in the Morita theory of (dg-)algebras. It was originally proved by Happel \cite{Ha89} and Buchweitz \cite{Bu03}, and later generalized to dg-algebras by Keller \cite{Keller03}.
\begin{lemma}(\cite[Section 4.5]{Keller03})\label{biCart}
Let $G$ be the dg-algebra of upper-triangular matrices and $C^\bullet(A,X,B)$ be the complex defined in Section \ref{subsec:HHAinfty}. Recall the morphisms $\mu_A,\mu_B$ and $\alpha,\beta$ defined in Section
\ref{subsec:HHAinfty}.
There is an exact triangle
\[
\xymatrix{
CC^\bullet(G,G)\ar[rr]^{\mu^*_A\oplus\mu^*_B} && CC^\bullet(A,A)\oplus CC^\bullet(B,B)\ar[rr]^{\alpha+\beta} && C^\bullet(A,X,B) .
}\]
In particular, we have a long exact sequence of Hochschild cohomology groups:
\[
\xymatrix{
\ldots\ar[r]&HH^n(G,G)\ar[r] & HH^n(A,A)\oplus HH^n(B,B)\ar[r] &\\ 
\Hom_{\D(A\otimes B^{\rm{op}})}(X,X[n])\ar[r] &\ldots
}
\]
\end{lemma}

\begin{proof}[\bf{Proof of Theorem \ref{mainthm}}]
Let $X$ be the germ of the hypersurface singularity defined by $W\in\CC[[x_1,\ldots,x_4]]$.
It was proved by Dyckerhoff that $k^{\st}$ is a compact generator of $\MF^\infty(W)$ (Corollary 4.12 \cite{Dyc09}). And by Proposition 5.10 of \cite{IW10}, $N^{\st}$ is a compact generator of $\MF^\infty(W)$. For simplicity, we denote $\Ac$ by $A$ and $B_W$ by $B$.  By the Morita theory of dg-categories (Theorem 3.11 \cite{Keller06}), the $A\otimes B^{\rm{op}}$ module $X:=\Hom_{\MF(W)}(k^{\st},N^{\st})$ induces an equivalence of triangulated categories:
\[
-\otimes_{A}^L X: \D A\cong \D B.
\]
This implies that $\alpha$ and $\beta$ are quasi-isomorphisms. By the exact triangle in Lemma \ref{biCart}, $\mu_A^*$ and $\mu_B^*$ are both quasi-isomorphisms. The map 
\[
\phi_X=\mu_A^*\circ(\mu_B^*)^{-1}
\] defines an isomorphism $HH^\bullet(A,A)\cong HH^\bullet(B,B)$. By Proposition \ref{mumap}, $\phi_X(m^B)=m^A$. If we identify $HH^\bullet(B,B)$ with the Milnor algebra $M_W$, then $m_B=W$ by Corollary \ref{m=W}. So $m_A=\phi_X(W)$.

Beginning with the contraction algebra $A=\Ac$, we reconstruct the Milnor algebra $M_W$ as the Hochschild cohomology $HH^\bullet(A,A)$, and the class of $W$ by $m_A$. Then by Theorem \ref{MY}, the germ of the hypersurface singularity is determined.
\end{proof}
\begin{remark}
We believe that the contraction algebra (with an appropriate $A_\infty$-enhancement) should determine the germ of the singularity underlying a flopping contraction of arbitrary dimension with one dimensional fiber. For higher dimensional flopping contraction with $Y$ being a hypersurface, we can show that Theorem \ref{inftycontr} still holds. However, the Milnor algebra should be replaced by the twisted de Rham complex $(\Omega^\bullet_V,dW\wedge)$. 
\end{remark}

\end{document}